\documentclass[final,5p,times,twocolumn]{elsarticlehack}

\usepackage{amsmath, amsthm, amssymb}
\usepackage{xcolor}

\usepackage[breaklinks]{hyperref}
\usepackage[hyphenbreaks]{breakurl}

\theoremstyle{plain}
\theoremstyle{definition}
\newtheorem{theorem}{Theorem}
\newtheorem{proposition}[theorem]{Proposition}
\newtheorem{corollary}[theorem]{Corollary}

\newtheorem{definition}[theorem]{Definition}

\usepackage{lineno}

\DeclareMathOperator{\rank}{rank}
\DeclareMathOperator{\trace}{trace}
\DeclareMathOperator{\poly}{poly}
\DeclareMathOperator{\size}{size}

\newcommand*{\colorboxed}{}
\def\colorboxed#1#{%
  \colorboxedAux{#1}%
}
\newcommand*{\colorboxedAux}[3]{%
  \begingroup
    \colorlet{cb@saved}{.}%
    \color#1{#2}%
    \boxed{%
      \color{cb@saved}%
      #3%
    }%
  \endgroup
}

\journal{Operations Research Letters}

\begin{document}

\begin{frontmatter}
\title{On Sparse Reflexive Generalized Inverse}


\author[UFRJ]{Marcia Fampa}
\ead{fampa@cos.ufrj.br}
\author[UMich,Corr]{Jon Lee}
\ead{jonxlee@umich.edu}

\address[UFRJ]{Universidade Federal do Rio de Janeiro}
\address[UMich]{University of Michigan, Ann Arbor, MI, USA}

\fntext[Corr]{Corresponding author: Jon Lee, Industrial and Operations Engineering Department
1205 Beal Avenue
University of Michigan
Ann Arbor, MI 48109-2117
USA. Email: jonxlee@umich.edu}

\begin{abstract}
We study sparse generalized inverses $H$ of a rank-$r$ real matrix $A$.
We give a construction for reflexive generalized inverses
having at most $r^2$ nonzeros. For $r=1$ and for $r=2$ with $A$ nonnegative, we demonstrate how to
minimize the (vector) 1-norm over reflexive generalized inverses.
For general $r$, we efficiently
find reflexive generalized inverses with 1-norm within approximately a factor of $r^2$ of
the minimum 1-norm generalized inverse. 
\end{abstract}


\end{frontmatter}


\section{Introduction}\label{sec:intro}

Generalized inverses are a key tool in matrix algebra and its applications.
In particular, the Moore-Penrose (M-P) pseudoinverse can be used to calculate the least-squares solution of an over-determined system of linear equations and the 2-norm minimizing solution of  an under-determined  system of linear equations.
Sparse optimization aims at finding sparse solutions of
optimization problems, often for computational efficiency in the use of the output of the optimization.
There is a tradeoff between finding an optimal solution and a sub-optimal  sparse solution, and sparse optimization
focuses on balancing the two.
In this spirit, \cite{dokmanic} developed left and right sparse pseudoinverses. \cite{FFL2016} introduced and analyzed various
sparse generalized inverses based on relaxing some of the ``M-P properties''.

{\bf Notation:} In what follows, for succinctness, we use vector-norm notation on matrices; so we write $\|H\|_1$ to mean
$\|\mbox{vec}(H)\|_1$ and $\|H\|_{\mbox{\scriptsize max}}$ to mean
$\|\mbox{vec}(H)\|_{\mbox{\scriptsize max}}$.
We use $I$ for an identity matrix and $J$ for a square all-ones matrix,
sometimes indicating the order with a subscript. Matrix dot product
is indicated by $\langle \cdot,\cdot \rangle$.

\subsection{Pseudoinverses}\label{sec:pseudo}
When a real matrix $A\in\mathbb{R}^{m\times n}$ is not square or not invertible,
we consider  generalized inverses  of $A$ (see \cite{rao1971}).
The most well-known generalized inverse  is the \emph{M-P pseudoinverse},
independently discovered by  E.H. Moore and R. Penrose.
If $A=U\Sigma V^\top$ is the real singular value decomposition of $A$ (see \cite{GVL1996}, for example),
then the
M-P pseudoinverse of $A$ can be defined as $A^+:=V\Sigma^+ U^\top$, where $\Sigma^+$
has the shape of the transpose of the diagonal matrix $\Sigma$, and is derived from $\Sigma$
by taking reciprocals of the non-zero (diagonal) elements of $\Sigma$ (i.e., the non-zero
singular values of $A$).
The M-P pseudoinverse, a central object in matrix theory, has many
concrete uses.

We seek to define different tractable sparse  generalized inverses, based on the
 the following fundamental characterization of the
M-P pseudoinverse.

\begin{theorem}[\cite{Penrose}]
For $A$ $\in$ $R^{m \times n}$, the M-P pseudoinverse $A^+$ is the unique $H$ $\in$ $\mathbb{R}^{n \times m}$ satisfying:
	\begin{align}
		& AHA = A \label{property1} \tag{P1}\\
		& HAH = H \label{property2} \tag{P2}\\
		& (AH)^{\top} = AH \label{property3} \tag{P3}\\
		& (HA)^{\top} = HA \label{property4} \tag{P4}
	\end{align}
\end{theorem}
So, the unique  $H$ satisfying  \ref{property1}$+$\ref{property2}$+$\ref{property3}$+$\ref{property4}
is the \emph{M-P pseudoinverse}.

We note that not all of the  M-P properties are needed for a generalized inverse to exactly solve key problems:

\begin{proposition}[see {\cite{FFL2016}}]\label{cor:lss}
If $H$ satisfies \ref{property1} and \ref{property3}, 
 then $x:=Hb$ (and of course $A^+b$)
solves  $\min\{\|Ax-b\|_2 ~:~ x\in\mathbb{R}^n\}$.
\end{proposition}

\begin{proposition}[see {\cite{FFL2016}}]\label{cor:minnorm}
If $H$ satisfies \ref{property1} and \ref{property4}, and $b$ is in the column space of $A$,
then $Hb$ (and of course $A^+b$) solves $\min\{\|x\|_2 ~:~ Ax=b,~ x\in\mathbb{R}^n\}$.
\end{proposition}

\subsection{Generalized inverses}\label{sec:generalized}

 We are interested in sparse matrices, so \ref{property1}
 is particularly important to enforce, because the zero-matrix
 (completely sparse) always satisfies \ref{property2}$+$\ref{property3}$+$\ref{property4}.
 Following \cite{RohdeThesis}, we call:
\begin{itemize}
\item any  $H$
satisfying \ref{property1}
a \emph{generalized inverse};
\item any  $H$
satisfying \ref{property1}$+$\ref{property2}
a \emph{reflexive generalized inverse};
\item any $H$
satisfying \ref{property1}$+$\ref{property2}$+$\ref{property3}
a \emph{normalized generalized inverse}.
 \end{itemize}

By Proposition \ref{cor:lss}, every normalized generalized inverse
solves least-squares problems.

We have the following key property.
\begin{theorem}[{\cite[Theorem 3.14]{RohdeThesis}}]
If $H$ is a generalized inverse of $A$, then
$\rank(H)\geq \rank(A)$. Moreover, a generalized inverse $H$ of $A$ is a reflexive
generalized inverse if and only if $\rank(H)= \rank(A)$.
\end{theorem}

\section{Sparse generalized inverses and a construction} \label{sec:sparseginv}

We are interested in sparse generalized inverses.
Finding a generalized inverse (i.e., solution of  \ref{property1}) with the minimum number of nonzeros
subject to various subsets of \ref{property2}, \ref{property3}, and \ref{property4} (but not all of them) is hard. So we take the standard approach of minimizing $\|H\|_1$ to induce sparsity, subject to  \ref{property1} and
various subsets of \ref{property2}, \ref{property3}, and \ref{property4} (but not all of them). From this point of view, we see that   \ref{property1}, \ref{property3} and \ref{property4} are linear constraints, hence easy to handle, while \ref{property2} is non-convex quadratic, and so rather nasty. Because of this, we are particularly interested in situations where a 1-norm minimizing generalized inverse $H$ is a reflexive generalized inverse; that is, when
\[
\min\left\{\|H\|_1 ~:~  \ref{property1}\right\} ~=~ \min\left\{\|H\|_1 ~:~  \ref{property1}+\ref{property2}\right\}.
\]

First, we give a recipe for constructing a somewhat-sparse generalized inverse of $A$. In particular, if $\rank(A)=r$, then we construct a \emph{reflexive} generalized inverse of $A$ having at most $r^2$ nonzeros.

\begin{theorem} \label{rbyrsol}
For $A\in\mathbb{R}^{m\times n}$, let $r:=\rank(A)$.
Let $\tilde{A}$ be \emph{any} $r\times r$ nonsingular submatrix of $A$.
Let $H\in\mathbb{R}^{n\times m}$ be such that its submatrix that
corresponds in position to that of $\tilde{A}$ in $A$ is equal to
$\tilde{A}^{-1}$, and other positions in $H$ are zero.
Then $H$ is a reflexive generalized inverse of $A$.
\end{theorem}

\begin{proof}
Without loss of generality,
we may assume that $\tilde{A}$ is the north-west block of $A$, so
\begin{equation*}
A=\left(\begin{array}{cc} \tilde{A}&B\\C&D\end{array}\right),
\end{equation*}
and
\begin{equation*}
H=\left(\begin{array}{cc} \tilde{A}^{-1}&0\\0&0\end{array}\right).
\end{equation*}

It is a simple computation to verify that $HAH=H$.
Furthermore,
\[
AHA= \left(\begin{array}{cc} \tilde{A}&B\\C&C\tilde{A}^{-1}B\end{array}\right).
\]
We have that $\rank(A) = \rank(\tilde{A})+ \rank(D-C\tilde{A}^{-1}B)$, where  $D-C\tilde{A}^{-1}B$  is the Schur complement of $\tilde{A}$ in $A$.
Therefore $\rank(D-C\tilde{A}^{-1}B)=0$, and so
$D=C\tilde{A}^{-1}B$, and therefore,  $AHA=A$.
So, we conclude that $H$ is a reflexive generalized inverse of $A$.
\end{proof}

\section{Exact solution}

It is useful to have an explicit formulation of
$\min\left\{\|H\|_1 ~:~  \ref{property1}\right\}
= \min\left\{\|H\|_1 ~:~  AHA=A \right\}$
as a linear-optimization problem (P) and its dual (D):

\begin{align*}
(P) \qquad &\min   \left\langle J , T \right\rangle~;  & \mbox{\underline{dual var.}} \\
& \mbox{subject to:}\\
             &T - H \geq 0~;\quad & U\\
		     &T + H \geq 0~;\quad & V\\
		     &AHA=A~.\quad & W\\[10pt]
(D) \qquad &\max   \left\langle A, W \right\rangle~;   \\
& \mbox{subject to:}\\
          &U + V = J~; \quad & T\\
          &-U + V + A^\top WA^\top= 0~; \quad & H\\
		  &U,V\geq 0~.
\end{align*}

More compactly, we can see (D) also as:
\begin{equation*}
\max\left\{ \left\langle A, W \right\rangle ~:~ \| A^\top WA^\top \|_{\mbox{\scriptsize max}} \leq 1 \right\}.
\end{equation*}

\subsection{Rank 1}

Next, we demonstrate that when $\rank(A)=1$,
construction of a 1-norm minimizing generalized inverse
can be based on our construction from Theorem
\ref{rbyrsol}.
So, in this case, a 1-norm minimizing generalized inverse can be chosen to be a
reflexive generalized inverse.

\begin{theorem}
\label{thmr1}
Let $A$ be an arbitrary rank-1 matrix,
which is, without loss of generality, of the form $A:=rs^\top$, where ${\bf 0}\not=r\in\mathbb{R}^m$ and ${\bf 0}\not=s\in\mathbb{R}^n$. For all $j\in M:=\{1,\ldots,m\}$ and $i\in N:=\{1,\ldots,n\}$, if $a_{ji}$ $(=r_js_i)$ $\neq 0$,  the matrix
\[
H := \frac{1}{ r_{j}s_{i} } e_ie^\top_j~,
\]
where $e_i\in \mathbb{R}^n$ and $e_j\in \mathbb{R}^m$ are standard unit vectors, is a reflexive generalized inverse of $A$.
Furthermore, if  $i$ and $j$ are selected respectively as $i^*:=\mbox{argmax}_{i\in N} \{|s_i|\}$  and  $j^*:=\mbox{argmax}_{j\in M}\{|r_j|\}$, then $H$ is a generalized inverse of $A$ with minimum 1-norm.
\end{theorem}

\begin{proof}
By Theorem  \ref{rbyrsol},
 $H$ satisfies   \ref{property1} and \ref{property2} for every choice of $i$ and $j$.
To prove that the minimum 1-norm of $H$ is attained when $i:=i^*$ and $j:=j^*$, we consider the following linear-optimization problem, any solution of which is
 is a generalized inverse of $A$ with minimum 1-norm.

\begin{align}
(P) \;\;\;&z :=  \min   \sum_{i\in N} \sum_{j\in M} t_{ij}   \\
& \mbox{subject to:}\\
             &t_{ij} - h_{ij} \geq 0~, \;\; i\in N,\ j\in M~; \label{primal_abs1}\\
		&t_{ij} + h_{ij} \geq 0~, \;\; i\in N,\ j\in M~;  \label{primal_abs2}\\
		&\sum_{k\in N}\sum_{\ell \in M}a_{pk}a_{\ell q}h_{k\ell}= a_{pq}~,\;\; p\in M,\ q\in N~.  \label{primal_other}
\end{align}

 The dual of ($P$) is:
\begin{align}
(D)\;\;\;&z :=  \max   \sum_{p\in M} \sum_{q\in N} a_{pq}w_{pq}   \\
& \mbox{subject to:}\\
             &u_{ij} + v_{ij} = 1~, \;\; i\in N,\ j\in M~; \label{dual_sum1}\\
             &-u_{ij} + v_{ij} + \sum_{p\in M}\sum_{q\in N}a_{pi}a_{jq}w_{pq}=0~, \;\; i\in N,\ j\in M~; \label{dual_other}\\
		&u_{ij}, v_{ij} \geq 0~, \;\; i\in N,\ j\in M~.
\end{align}

A feasible solution for ($P$) is
\[
\begin{array}{lccl}
t_{i^* j^*} &=& \frac{1}{ |r_{j^*}| |s_{i^*}| }~, &\\
t_{ij}&=&0~, &      \mbox{if } i\neq i^* \mbox{ or } j\neq j^*~ (i\in N,\ j\in M)~, \\
h_{i^* j^*} &=& \frac{1}{ r_{j^*}s_{i^*} }~, &\\
h_{ij}&=&0~, &      \mbox{if } i\neq i^* \mbox{ or } j\neq j^*~ (i\in N,\ j\in M)~.
\end{array}
\]
Clearly, $t_{ij}= |h_{ij}|$, for $i\in N,\ j\in M$.
Therefore, (\ref{primal_abs1}--\ref{primal_abs2}) is satisfied.
The left-hand side of \eqref{primal_other} simplifies to
\begin{equation*}
a_{p i^*}a_{j^* q}h_{ i^* j^*} = r_p s_{i^*} r_{j^*} s_q \frac{1}{ r_{j^*}s_{i^*} }
= r_p s_q = a_{pq}~,
\end{equation*}
Therefore \eqref{primal_other} is satisfied.
Hence the solution is primal feasible.

A feasible solution for ($D$) can be  obtained by setting:
\[
\begin{array}{ll}
u_{ij}=\frac{1}{2}\left(1+\frac{r_js_i}{|r_{j^*}||s_{i^*}|}\right)~,&  i\in N,\ j\in M~,\\
v_{ij}=1 -u_{ij}~,&  i\in N,\ j\in M~,\\
w_{j^* i^*} = \frac{1}{ r_{j^*}|r_{j^*}| s_{i^*}|s_{i^*}| }~, &\\
w_{ji} =0~, &      \mbox{if } i\neq i^* \mbox{ or } j\neq j^*~ (i\in N,\ j\in M)~.
\end{array}
\]
It is easy to see that $0\leq u_{ij}\leq 1$, because $-1 \leq \frac{r_js_i}{|r_{j^*}||s_{i^*}|} \leq 1$, and hence $v_{ij}\geq 0$ (for $ i\in N,\ j\in M$). We have

\begin{equation*}
v_{ij} - u_{ij} ~=~ 1 - 2 u_{ij} ~=~
-\frac{r_js_i}{|r_{j^*}||s_{i^*}|}~.
\end{equation*}
Next, the left-hand side of \eqref{dual_other} simplifies to
\begin{equation*}
-\frac{r_js_i}{|r_{j^*}||s_{i^*}|} ~+~ a_{j^* i}a_{j i^*}w_{ j^* i^*}
=
-\frac{r_js_i}{|r_{j^*}||s_{i^*}|} + r_{j^*} s_i r_j s_{i^*} \frac{1}{ r_{j^*}|r_{j^*}| s_{i^*}|s_{i^*}| }= 0.
\end{equation*}
Therefore \eqref{dual_other} is satisfied.
Hence the solution is dual feasible.

The objective value of our dual solution is
\[
a_{j^* i^*} w_{ j^* i^*}
=
r_{j^*} s_{i^*}  \frac{1}{ r_{j^*}|r_{j^*}| s_{i^*}|s_{i^*}| }
=
 \frac{1}{|r_{j^*}| |s_{i^*}|},
\]
which is the objective value of our primal solution.

Therefore, the result follows from the weak-duality theorem of linear optimization.
\end{proof}

\subsection{Rank 2}

Generally, when $\rank(A)=2$, we cannot base a 1-norm minimizing generalized inverse
on our construction from Theorem \ref{rbyrsol}. For example,
with
\[
A := \left(
      \begin{array}{rr}
        1 & 1 \\
        1 & -1 \\
        2 & 0 \\
      \end{array}
    \right),
\]
we have a  1-norm minimizing generalized inverse
\[
H := \left(
      \begin{array}{ccc}
        0 & 0 & \frac{1}{2} \\
        \frac{1}{2} & -\frac{1}{2} & 0\\
      \end{array}
    \right).
\]
This $H$ is reflexive as well (because $A$ has full column rank).
We have $\|H\|_1=\frac{3}{2}$, while all (three) of the
(reflexive) generalized inverses based on the
construction from Theorem \ref{rbyrsol} have 1-norm equal to 2.

Next, we demonstrate that \emph{under a natural but restrictive condition},
specifically if some rows and columns of $A$ can be multiplied by $-1$ so that
the matrix becomes nonnegative,
when $\rank(A)=2$, construction of  a 1-norm minimizing generalized inverse
can be based on our construction from Theorem
\ref{rbyrsol}.
So, in this case of  $A\geq 0$ and $\rank(A)=2$, a 1-norm minimizing generalized inverse can be chosen
to be a
reflexive generalized inverse. We note that the rank-2 example above
does not satisfy the condition of our theorem below.


\begin{theorem}
\label{rank2sol}
Let $A\in\mathbb{R}^{m\times n}$ be a rank-2 matrix
such that some rows and columns of $A$ can be multiplied by $-1$ so that
the matrix becomes nonnegative.
Let $\tilde{A}$ be a nonsingular $2\times 2$ submatrix of
$A$ that minimizes $\|\tilde{A}^{-1}\|_1$ among all such submatrices.
Construct $H$ as per Theorem \ref{rbyrsol}.
Then $H$ is a generalized inverse of $A$ with minimum 1-norm.
Moreover, it is a reflexive generalize inverse of $A$.
\end{theorem}

\begin{proof}
First, we demonstrate that without loss of generality, we can assume that $A\geq 0$.
Let $L$ and $R$ be diagonal matrices, with all diagonal entries equal to $\pm1$,
so that $\hat{A}:=LAR\geq 0$. Let $\hat{H}$ be any generalized inverse of $\hat{A}$.
Then $\hat{A}\hat{H}\hat{A}=\hat{A}$
$\Leftrightarrow$
$(LAR)\hat{H}(LAR)=LAR$
$\Leftrightarrow$
$L(LAR\hat{H}LAR)R=L(LAR)R$
$\Leftrightarrow$
$A(R\hat{H}L)A=A$.
So $\hat{H}$ is a generalized inverse of $\hat{A}$
$\Leftrightarrow$
$H:=R\hat{H}L$ is a generalized inverse of $A$.
Because $\| H\|_1 = \| \hat{H}\|_1$,
we can equivalently work with $A$ or $\hat{A}$. We can further see that
$\hat{H}\hat{A}\hat{H}=\hat{H}$ $\Leftrightarrow$
$\hat{H}(LAR)\hat{H} = \hat{H}$
$\Leftrightarrow$ $R(\hat{H}LAR\hat{H})L = R\hat{H}L$
$\Leftrightarrow$ $HAH = H$. Therefore, $\hat{H}\hat{A}\hat{H}=\hat{H}$ $\Leftrightarrow$
$HAH=H$.
So again, when considering \emph{reflexive} generalized inverses,
we can equivalently work with $A$ or $\hat{A}$. So, in what follows, we can assume that
$A\geq 0$.

We assume without loss of generality that
$\tilde{A}$ is in the north-west corner of $A$,
So we take $A$ to have the form
\[
\left(\begin{array}{cc} \tilde{A}&B \\C &D\end{array}\right).
\]
We may further assume without loss of generality that $\det(\tilde{A})>0$,
interchanging  rows  if necessary.

Let
\begin{equation*}
W:=\left(\begin{array}{cc} \tilde{W}&0\\0&0\end{array}\right)
:= \left(\begin{array}{cc} \tilde{A}^{-\top}(2I-J)\tilde{A}^{-\top}&0\\0&0\end{array}\right).
\end{equation*}

First, we wish to establish that the objective value of $H$ in (P) is
the same as the objective value of $W$ in (D). As we know,
\[
\tilde{A}^{-1}=\frac{1}{\det (\tilde{A})}
\left(\begin{array}{rr} a_{22}&-a_{12}\\-a_{21}&a_{11}\end{array}\right).
\]
We have
\[
A^\top W = \left(\begin{array}{cc} \tilde{A}^\top&C^ \top\\B^\top&D^\top\end{array}\right) \left(\begin{array}{cc} \tilde{W}&0\\0&0\end{array}\right)= \left(\begin{array}{cc} \tilde{A}^\top\tilde{W} &0\\B^\top\tilde{W}&0\end{array}\right).
\]
Note that
$\tilde{A}^\top \tilde{W} = (2I-J)\tilde{A}^{-\top}$. Therefore,
\begin{align*}
\left\langle A,W\right\rangle &=~ \trace (A^\top W)= \trace ( \tilde{A}^\top \tilde{W}) \\
&=~ \trace ((2I-J)\tilde{A}^{-\top})\\
& =~ 2\trace (\tilde{A}^{-\top}) - \trace (J\tilde{A}^{-\top})\\
&=~ 2\frac{a_{11}+a_{22}}{\det (\tilde{A})} - \frac{a_{11}-a_{12}-a_{21}+a_{22}}{\det (\tilde{A})}\\
&=~  \frac{a_{11}+a_{12}+a_{21}+a_{22}}{\det (\tilde{A})}\\
&=~ ||\tilde{A}^{-1}||_1 ~=~ ||H||_1~.
\end{align*}

Next, we will check the dual feasibility of $W$,
which amounts to $||A^\top WA^\top||_{\max}\leq 1$. We have
\begin{align*}
A^\top WA^\top
&= \left(\begin{array}{cc} \tilde{A}^\top&C^ \top\\B^\top&D^\top\end{array}\right)
   \left(\begin{array}{cc} \tilde{W}&0\\0&0\end{array}\right)
   \left(\begin{array}{cc} \tilde{A}^\top&C^ \top\\B^\top&D^\top\end{array}\right) \\
&= \left(\begin{array}{cc} \tilde{A}^\top \tilde{W}\tilde{A}^\top&\tilde{A}^\top \tilde{W}C^\top\\
                            B^\top \tilde{W}\tilde{A}^\top&B^\top \tilde{W}C^\top\end{array}\right)\\
&= \left(\begin{array}{cc} 2I-J&(2I-J)\tilde{A}^{-\top}C^\top\\
                            B^\top\tilde{A}^{-\top}(2I-J)  &B^\top\tilde{A}^{-\top}(2I-J)\tilde{A}^{-\top}C^\top  \end{array}\right).
\end{align*}

Clearly $\|2I-J\|_{\mbox{\scriptsize max}} \leq 1$, therefore the north-west block of $A^\top WA^\top$
meets the dual-feasibility condition.

Next we consider $(2I-J)\tilde{A}^{-\top}C^\top$, the north-east block of $A^\top WA^\top$.
It suffices to check that
$\| (2I-J)\tilde{A}^{-\top}\gamma \|_{\mbox{\scriptsize max}} \leq 1$,
where $\gamma$ is an arbitrary column of $C^\top$. We may as well assume that
$\gamma$ is not all zero, because in that case we certainly have
$\| (2I-J)\tilde{A}^{-\top}\gamma \|_{\mbox{\scriptsize max}} \leq 1$.

We have
\[
(2I-J)\tilde{A}^{-\top}\gamma
 ~=~ \frac{1}{\det (\tilde{A})}
 \left(\begin{array}{rr} 1&-1\\-1&1\end{array}\right)
\left(\begin{array}{rr} a_{22}&-a_{21}\\
                       -a_{12}&a_{11}\end{array}\right)\left(\begin{array}{r} \gamma_1\\
                       \gamma_2\end{array}\right).
\]
Employing Cramer's rule, it is easy to check that this is equal to
\[
\frac{1}{\det(\tilde{A})}
 \left(\begin{array}{r} \det (\tilde{A}_{\gamma/1}) - \det (\tilde{A}_{\gamma/2})\\
                        \det (\tilde{A}_{\gamma/2}) - \det (\tilde{A}_{\gamma/1}) \end{array}\right),
\]
where
\[
\tilde{A}_{\gamma/1}:=\left(\begin{array}{rr} \gamma_1&\gamma_2\\ a_{21}&a_{22}\end{array}\right) ~\mbox{ and }~
\tilde{A}_{\gamma/2}:=\left(\begin{array}{rr} a_{11}&a_{12}\\ \gamma_1 &\gamma_2\end{array}\right).
\]

%
%

As $\rank(A)=2$  and $\tilde{A}$ is nonsingular, we have that
\[
(\gamma_1,\gamma_2)= \alpha (a_{11}, a_{12}) + \beta (a_{21}, a_{22}),
\]
where
\begin{equation*}
\label{dets}
\begin{array}{ll}
&\alpha = \det (\tilde{A}_{\gamma/1}) / \det (\tilde{A}),\\
&\beta = \det (\tilde{A}_{\gamma/2}) /  \det (\tilde{A}),
\end{array}
\end{equation*}
and we do not have $\alpha=\beta=0$ (because we have seen that we can assume
that we do not have $\gamma_1=\gamma_2=0$, and we have that $\tilde{A}$ is nonsingular).

In summary, we have
\begin{equation}
\label{alphabetadiff}
(2I-J)\tilde{A}^{-\top}\gamma
 ~=~
 \left(\begin{array}{c} \alpha-\beta\\ \beta-\alpha \end{array}\right).
\end{equation}

Considering that
$\tilde{A}$ is chosen
to  minimize $||\tilde{A}^{-1}||_1$ among all nonsingular $2\times 2$ submatrices of $A$,  we  have that
\begin{equation}
\label{1norm-ineq}
 ||\tilde{A}^{-1}||_1\leq
||\tilde{A}_{\gamma/i}^{-1}||_1
~,
\end{equation}
whenever $\tilde{A}_{\gamma/i}$ is nonsingular,
for $i=1,2$.

%

\smallskip
\noindent {\bf Claim $\mathcal{A}_+$:~ $\alpha>0$ $\Rightarrow$ $\alpha - \beta \leq 1$.}
\begin{proof}
Because we assume $A\geq 0$,
\eqref{1norm-ineq} for $i=1$ becomes
\[
\frac{a_{11}+a_{12}+a_{21}+a_{22}}{\det (\tilde{A})}\leq \frac{a_{21}+a_{22}+ (\alpha a_{11} + \beta a_{21})+ (\alpha a_{12} + \beta a_{22})}{\alpha\det (\tilde{A})}.
\]
Simplifying, we obtain
\[
(\alpha-\beta-1)(a_{21}+a_{22}) \leq 0,
\]
which implies $\alpha -\beta \leq 1$.
\end{proof}

\noindent {\bf Claim $\mathcal{A}_-$:~ $\alpha<0$ $\Rightarrow$ $\alpha - \beta < 1$.}
\begin{proof}
Now \eqref{1norm-ineq} for $i=1$ becomes
\[
\frac{a_{11}+a_{12}+a_{21}+a_{22}}{\det (\tilde{A})}\leq \frac{a_{21}+a_{22}+ (\alpha a_{11} + \beta a_{21})+ (\alpha a_{12} + \beta a_{22})}{-\alpha\det (\tilde{A})}.
\]
Simplifying, we obtain
\[
(1+\alpha+\beta)(a_{21}+a_{22}) \geq -2\alpha(a_{11}+a_{12}).
\]
Now, adding $-2\alpha(a_{21}+a_{22})$ to both sides, we obtain
\[
(1-\alpha+\beta)(a_{21}+a_{22}) \geq -2\alpha(a_{11}+a_{12}+a_{21}+a_{22}),
\]
the right-hand side of which is positive, and so
we conclude that $\alpha-\beta< 1$.
\end{proof}

\noindent {\bf Claim $\mathcal{B}_+$:~ $\beta>0$ $\Rightarrow$ $\beta-\alpha \leq 1$.}
\begin{proof}
Now \eqref{1norm-ineq} for $i=2$ becomes
\[
\frac{a_{11}+a_{12}+a_{21}+a_{22}}{\det (\tilde{A})}\leq
\frac{a_{11}+a_{12}+ (\alpha a_{11} + \beta a_{21})+ (\alpha a_{12} + \beta a_{22})}{\beta\det (\tilde{A})}.
\]
Simplifying, we obtain
\[
(1+\alpha-\beta)(a_{11}+a_{12}) \geq 0,
\]
which implies $\beta-\alpha\leq 1$.
\end{proof}

\smallskip
\noindent {\bf Claim $\mathcal{B}_-$:~ $\beta<0$ $\Rightarrow$ $\beta-\alpha < 1$.}
\begin{proof}
Now \eqref{1norm-ineq} for $i=2$ becomes
\[
\frac{a_{11}+a_{12}+a_{21}+a_{22}}{\det (\tilde{A})}\leq
\frac{a_{11}+a_{12}+ (\alpha a_{11} + \beta a_{21})+ (\alpha a_{12} + \beta a_{22})}{-\beta\det (\tilde{A})}.
\]
Simplifying, we obtain
\[
(1+\alpha+\beta)(a_{11}+a_{12}) \geq -2\beta(a_{21}+a_{22}).
\]
Now, adding $-2\beta(a_{11}+a_{12})$ to both sides, we obtain
\[
(1+\alpha-\beta)(a_{11}+a_{12}) \geq -2\beta(a_{21}+a_{22}+a_{11}+a_{12}),
\]
the right-hand side of which is positive, and so
we conclude that $\beta-\alpha< 1$.
\end{proof}

\smallskip
\noindent {\bf Claim $\mathcal{A}_0$:~ $\alpha=0$ $\Rightarrow$ $0< \beta < 1$.}
\begin{proof}
Certainly $\alpha=0$ $\Rightarrow$ $\beta>0$, because $A\geq 0$.
We can further apply Claim $\mathcal{B}_+$ and conclude that
$\beta = \beta-\alpha \leq 1$~.
\end{proof}

\smallskip
\noindent {\bf Claim $\mathcal{B}_0$:~ $\beta=0$ $\Rightarrow$ $0< \alpha < 1$.}
\begin{proof}
Certainly $\beta=0$ $\Rightarrow$ $\alpha>0$, because $A\geq 0$.
We can further apply Claim $\mathcal{A}_+$ and conclude that
$\alpha = \alpha-\beta \leq 1$~.
\end{proof}


\bigskip

Therefore, considering \eqref{alphabetadiff} and the six claims,  we have
$\| (2I-J)\tilde{A}^{-\top}\gamma \|_{\mbox{\scriptsize max}}\leq 1$,
and  so $\| (2I-J)\tilde{A}^{-\top}C^\top \|_{\mbox{\scriptsize max}}\leq 1$.

By symmetry, can see that  $B^\top\tilde{A}^{-\top}(2I-J)$, the south-west block of $A^\top WA^\top$, satisfies $\|B^\top\tilde{A}^{-\top}(2I-J)\|_{\mbox{\scriptsize max}}\leq 1$.

Finally,we consider $B^\top\tilde{A}^{-\top}(2I-J)\tilde{A}^{-\top}C^\top$,
 the south-east block of $A^\top WA^\top$. We have
 \begin{align*}
& \|B^\top\tilde{A}^{-\top}(2I - J)
\tilde{A}^{-\top}C^\top\|_{\mbox{\scriptsize max}}\\
&\quad= \frac{1}{2}  \| B^\top\tilde{A}^{-\top}(4I - 2J)\tilde{A}^{-\top}C^\top\|_{\mbox{\scriptsize max}}\\
&\quad= \frac{1}{2}  \| B^\top\tilde{A}^{-\top}(4I - 4J +J^2)\tilde{A}^{-\top}C^\top \|_{\mbox{\scriptsize max}}\\
&\quad= \frac{1}{2}  \| B^\top\tilde{A}^{-\top}(2I-J)(2I-J)
\tilde{A}^{-\top}C^\top\|_{\mbox{\scriptsize max}}\\
&\quad\leq \|B^\top\tilde{A}^{-\top}
(2I-J)\|_{\mbox{\scriptsize max}}\,
 \|  (2I-J)\tilde{A}^{-\top}C^\top\|_{\mbox{\scriptsize max}}  ~\leq~ 1
\end{align*}
So we have $\|B^\top\tilde{A}^{-\top}(2I - J)\tilde{A}^{-\top}C^\top\|_{\mbox{\scriptsize max}}\leq 1$,
and therefore we can finally conclude that our constructed $W$ is dual feasible.

Therefore, the result follows from the weak-duality theorem of linear optimization.

\end{proof}

It is natural to wonder whether we can go beyond rank 2, when $A\geq 0$.
But the following example indicates that there is no straightforward way to do this.
With
\[
A := \left(
      \begin{array}{ccc}
        2 & 1 & 0 \\
        0 & 2 & 1 \\
        1 & 2 & 0 \\
        2 & 1 & 1 \\
      \end{array}
    \right),
\]
we have a 1-norm minimizing generalized inverse
\[
H:= \left(
  \begin{array}{rrrr}
    \frac{1}{4} & -\frac{1}{4} & 0 & \frac{1}{4} \\[5pt]
    0 & \frac{1}{5} & \frac{2}{5} & -\frac{1}{5} \\[5pt]
    0 & \frac{3}{5} & -\frac{4}{5} & \frac{2}{5} \\
  \end{array}
\right).
\]
This $H$ is reflexive  (because $A$ has full column rank).
We have $\|H\|_1=3\frac{7}{20}$, while the four
(reflexive) generalized inverses based on the construction
from Theorem 5 have 1-norm equal to: 5, $4\frac{1}{4}$, 4, $3\frac{3}{5}$.

\section{Approximation}\label{sec:approx}

In this section, for general $r:=\rank(A)$, we demonstrate
how to efficiently find a reflexive generalized inverse following
our block construction that is within approximately a factor of $r^2$ of
the (vector) 1-norm of the generalized inverse having minimum (vector) 1-norm.

\begin{definition}
For $A\in \mathbb{R}^{m\times n}$, let $r:=\rank(A)$.
For $\sigma$ an ordered subset of $r$ elements from $\{1,\ldots,m\}$ and $\tau$ an ordered subset of $r$ elements from  $\{1,\ldots,n\}$,
let $A[\sigma,\tau]$ be the $r\times r$ submatrix of $A$ with
row indices $\sigma$ and column indices $\tau$.
For fixed $\epsilon \geq 0$,
if $|\det(A[\sigma,\tau])|$ cannot be increased by a factor of more than
$1+\epsilon$  by either swapping
an element of $\sigma$ with one from its complement
or swapping
an element of $\tau$ with one from its complement, then
we say that $A[\sigma,\tau]$ is a \emph{$(1+\epsilon)$-local maximizer
for the absolute determinant} on the set of $r\times r$ nonsingular submatrices of
$A$.
\end{definition}

\begin{theorem}\label{thm:approx}
For $A\in \mathbb{R}^{m\times n}$, let $r:=\rank(A)$. Choose $\epsilon\geq 0$,
and let $\tilde{A}$ be a $(1+\epsilon)$-local maximizer
for the absolute determinant on the set of $r\times r$ nonsingular submatrices of
$A$.
Construct $H$ as per Theorem \ref{rbyrsol}.
Then $H$ is a (reflexive) generalized inverse (having at most $r^2$
nonzeros), satisfying
$\| H \|_1 \leq r^2(1+\epsilon)^2 \| H_{\mbox{\scriptsize opt}} \|_1$,
where $H_{\mbox{\scriptsize opt}}$ is a 1-norm minimizing generalized inverse of $A$~.
\end{theorem}

\begin{proof}
We will construct a dual feasible solution with objective value
equal to $\frac{1}{r^2(1+\epsilon)^2} \| H \|_1$~. By weak duality for linear optimization, we have $\frac{1}{r^2(1+\epsilon)^2} \| H \|_1 \leq  \|H_{\mbox{\scriptsize opt}} \|_1$~.

As in the proof of Theorem \ref{rank2sol},
we assume without loss of generality that
$\tilde{A}$ is in the north-west corner of $A$,
So we take $A$ to have the form
\[
\left(\begin{array}{cc} \tilde{A}&B \\C &D\end{array}\right).
\]
But notice that we choose $\tilde{A}$ differently than in Theorem \ref{rank2sol}.
Again, we let
\begin{equation*}
W:=\left(\begin{array}{cc} \tilde{W}&0\\0&0\end{array}\right)
:= \left(\begin{array}{cc} \tilde{A}^{-\top}(2I-J)\tilde{A}^{-\top}&0\\0&0\end{array}\right).
\end{equation*}
We have already seen in the proof of Theorem \ref{rank2sol}
that $\left\langle A,W\right\rangle = \| H\|_1$~.
So, it suffices to demonstrate that $\| A^\top W A^\top \|_{\mbox{\scriptsize max}} \leq r^2(1+\epsilon)^2$ (then $\frac{1}{r^2(1+\epsilon)^2}W$ is dual feasible).

First we can easily see that
\[
\|\tilde{A}^\top \tilde{W} \tilde{A}^\top \|_{\mbox{\scriptsize max}}
=\| 2I-J \|_{\mbox{\scriptsize max}}
=1 \leq r^2(1+\epsilon)^2.
\]

Next, we consider
$\bar{\gamma}:=\tilde{A}^\top \tilde{W} \gamma = (2I-J)\tilde{A}^{-\top} \gamma$~,
where $\gamma$ is an arbitrary column of $C^\top$~.
By Cramers' rule, where $\tilde{A}_i(\gamma)$ is $\tilde{A}$ with
column $i$ replaced by $\gamma$, we have
\begin{align*}
\bar{\gamma}
=&~ (2I-J) \frac{1}{\det(\tilde{A})}
\left(
  \begin{array}{c}
   \det(\tilde{A}_1(\gamma)) \\
    \vdots \\
   \det(\tilde{A}_r(\gamma))  \\
  \end{array}
\right)\\
=&~ \frac{1}{\det(\tilde{A})}
\left(
  \begin{array}{c}
   \det(\tilde{A}_1(\gamma)) - \sum_{i\not=1} \det(\tilde{A}_i(\gamma)) \\
    \vdots \\
   \det(\tilde{A}_r(\gamma)) - \sum_{i\not=r} \det(\tilde{A}_i(\gamma)) \\
  \end{array}
\right).
\end{align*}
We have, for $j=1,\ldots,r$,
\[|\bar{\gamma}_j| ~\leq~
\sum_{i=1}^r \frac{|\det(\tilde{A}_i(\gamma))|}{|\det(\tilde{A})|} ~\leq~ r(1+\epsilon) ~\leq~ r^2(1+\epsilon)^2~,
\]
because $\tilde{A}$ is a
$(1+\epsilon)$-local maximizer
for the absolute determinant.

By symmetry, we can conclude that
\[
\|B^\top \tilde{W} \tilde{A}^\top\|_{\mbox{\scriptsize max}} \leq r(1+\epsilon) ~\leq~ r^2(1+\epsilon)^2.
\]

Finally, we have
\begin{align*}
&\| B^\top \tilde{W} \tilde{C}^\top\|_{\mbox{\scriptsize max}}
~=~
\| (B^\top \tilde{A}^{-\top})(2I-J)(\tilde{A}^{-\top} \tilde{C}^\top)\|_{\mbox{\scriptsize max}}\\
&\qquad\leq r^2~  \| B^\top \tilde{A}^{-\top}\|_{\mbox{\scriptsize max}}~~
\| 2I-J \|_{\mbox{\scriptsize max}}~~
\| \tilde{A}^{-\top} \tilde{C}^\top \|_{\mbox{\scriptsize max}} ~\leq~ r^2(1+\epsilon)^2~,
\end{align*}
using that  $\tilde{A}$ is a
$(1+\epsilon)$-local maximizer
for the absolute determinant
to conclude that  $ \| B^\top \tilde{A}^{-\top}\|_{\mbox{\scriptsize max}}\leq 1+\epsilon$ and $\| \tilde{A}^{-\top} \tilde{C}^\top \|_{\mbox{\scriptsize max}}\leq 1+\epsilon$, and noticing again that $\| 2I-J \|_{\mbox{\scriptsize max}}= 1$.
\end{proof}

We note that in Theorem \ref{thm:approx} we could have
required the stronger condition that $\tilde{A}$ globally
maximizes the absolute determinant among $r\times r$ nonsingular
submatrices of $A$. But we prefer our hypothesis, both because it
is weaker and because we can find an $\tilde{A}$ satisfying
our hypothesis by a simple finitely-terminating local search.
Moreover, if $A$ is rational and we choose $\epsilon$ \emph{positive},
then our local search is \emph{efficient}:

\begin{theorem}\label{thm:steps}
Let $A$ be rational. Choose a fixed $\epsilon>0$. Then
the number of steps of local search to reach a $(1+\epsilon)$-local maximizer
of the absolute determinant on the set of
$r\times r$ nonsingular submatrices of $A$ is
$\mathcal{O}(\poly(\size(A))) (1+\frac{1}{\epsilon})$.
\end{theorem}

\begin{proof}
Let
\[
\Delta:=\max \{
|\det(\tilde{A})| ~:~
\mbox{$\tilde{A}$ is an $r\times r$ nonsingular submatrix of $A$}
\}
\]
and
\[
\delta:=\min \{
|\det(\tilde{A})| ~:~
\mbox{$\tilde{A}$ is an $r\times r$ nonsingular submatrix of $A$}
\}.
\]
So the number of steps $k$ of local search
must satisfy $(1+\epsilon)^k \leq \Delta/\delta$~.
It is well known that the number of
bits to encode $\Delta$ and $\delta$ is polynomial
in the number of bits in a binary encoding of $A$ (see \cite{Lovasz}, for example). Hence
\[
k~=~\frac{\mathcal{O}(\poly(\size(A)))}{\log(1+\epsilon)}
~\leq~ \mathcal{O}(\poly(\size(A))) \left(1+ \frac{1}{\epsilon} \right).
\]
\end{proof}

Putting Theorems \ref{thm:approx} and \ref{thm:steps} together,
we get the following result.
\begin{corollary}
We have an FPTAS (fully polynomial-time approximation scheme; see \cite{WilliamsonShmoys})
for calculating a reflexive generalized inverse $H$ of $A$ that has
$\|H\|_1$ within a factor of $r^2$ of
$\| H_{\mbox{\scriptsize opt}} \|_1$~,
where $H_{\mbox{\scriptsize opt}}$ is a 1-norm minimizing generalized inverse of $A$~.
\end{corollary}

\section{Remarks} \label{sec:conc}
An interesting issue is whether the
analysis of our choice of $\tilde{A}$ in Theorem \ref{thm:approx} is
tight --- that is, for our choice of
$\tilde{A}$, can we establish a
sharper bound than $r^2$~? Also, is there a better choice of
$\tilde{A}$ in Theorem \ref{thm:approx}?

Another natural issue is whether we can go beyond rank-2 matrices
(see Theorem \ref{rank2sol}), and with what conditions on $A$?

\section*{Acknowledgments}
M. Fampa was supported in part by CNPq grant 303898/2016-0.
J. Lee was supported in part by ONR grant N00014-17-1-2296 and LIX, l'\'Ecole Polytechnique.

\section*{References}

\bibliography{ginv}
\bibliographystyle{plain}

\end{document}